%% file: bicat.tex
\documentclass[11pt,a4paper]{article}

\input{header}

\usepackage{framed}
\usepackage{boxedminipage}

\newcommand{\PSh}{\mathrm{PSh}}
\newcommand{\Fun}{\mathrm{Fun}}
\newcommand{\Map}{\mathrm{Map}}

\newcommand{\Cat}{\mathbf{Cat}}
\newcommand{\Op}{\mathbf{Op}}
\newcommand{\Cart}{\mathrm{Cart}}

\newcommand{\op}{\mathrm{op}}

\bibliography{bicat}
\title{The operad that co-represents enrichment}
\author{Andrew W. Macpherson}
\begin{document}

\maketitle

\abstract{I show that the theories of enrichment in a monoidal $\infty$-category defined by Hinich and by Gepner-Haugseng agree, and that the identification is unique. Among other things, this makes the Yoneda lemma available in the former model.}

\tableofcontents

\section{Introduction}

The notion of an enrichment of a category $\rm C$ in a monoidal category $(\sh V,\otimes)$ as a bifunctor
\[ \rm C_{\sh V}(-,-):\rm C^\op\times\rm C\rightarrow\sh V, \]
together with a composition law satisfying an associativity constraint, goes back almost as far as category theory itself.\footnote{For a potted history, see the introduction to \cite{Kelly}.} Trying to transplant this notion into higher category theory, one encounters the same difficulties controlling `coherent associativity' as one does when defining algebra objects --- that is, one-object enriched categories.

Recently, two approaches have appeared that provide a framework of $\sh V$-enrichments for any monoidal $\infty$-category $(\sh V,\otimes)$: the \emph{categorical algebras in $\sh V$} of \cite{Gepner_Haugseng} and the \emph{$\sh V$-enriched precategories} of \cite{Hin}. As in the better-known case of coherently associative algebras, the trick is in selecting a gadget that indexes operations --- in other words, a \emph{co-representing} object. In this case, the relevant object is a planar $\infty$-operad.

It is claimed in \cite{Hin} that its corepresenting operad $\mathtt{Ass}_X$ is `the same' as a simplicial multicategory $\sh O_X$ defined in \cite{Gepner_Haugseng}, and therefore that the attendant theories of $\sh V$-enrichments agree. However, although the author suggests the premise for the comparison, no complete justification of the claim appears. This note fills that gap: we show that:
\begin{itemize}\item The Gepner-Haugseng theory of categorical algebras is naturally (in $\sh V$) equivalent to Hinich's $\sh V$-enriched precategories; \item that the equivalence is unique --- the theory has no natural autoequivalences.\end{itemize}
In the rest of this introduction, I sketch out the broad strokes of the argument.

\begin{thm}\label{one}There is a unique equivalence, universal in $(\sh V,\otimes)$, between the Gepner-Haugseng category of categorical algebras in $\sh V$ and the Hinich category of $\sh V$-enriched precategories.\end{thm}
\begin{proof}The corepresenting planar operads in the two cases are as follows (for more details, see \S\ref{PRELIM}):
\begin{itemize}\item In \cite[\S4.2]{Gepner_Haugseng} the object is denoted $L_\rm{gen}\Delta^\op_X$; we will use the fact \cite[Cor.\ 4.2.8]{Gepner_Haugseng} that it is presented by the simplicial multicategory $\sh O_X$.
\item Hinich's $\mathtt{Ass}_X$ is defined \cite[\S3.2]{Hin} by specifying its spaces of simplices
\begin{align} \sigma:\Delta_\mathtt{Ass},\ X:\Cat \qquad \Fun_\mathtt{Ass}(\sigma,\mathtt{Ass}_X) \cong \Fun(\sh F(\sigma),X) \end{align}
in terms of an explicit functor $\sh F:\Delta_\mathtt{Ass}\rightarrow\Cat$ with domain the category of simplices of the associative operad.
\end{itemize}
In light of the definition of $\mathtt{Ass}_X$, it is enough to exhibit the $\infty$-functor associated to $\sh O$ as a right adjoint to a left Kan extension of $\sh F$:
\begin{align} \hat{\sh F}:\Cat_\mathtt{Ass}\rightleftarrows \Cat:\sh O.\end{align} This follows from proposition \ref{TECH_ADJOINT} and corollary \ref{QUILLEN_MAIN}.

By proposition \ref{UNIQUE_PROP}, $\Aut(\sh F)$ is trivial, whence the adjunction data is unique.
\end{proof}

\paragraph{Perspective}The uniqueness statement in theorem \ref{one} echoes the main results of \cite{BS-P} concerning $(\infty,n)$-categories --- in fact, it is even more satisfying, because it is not required to fix the inclusion of a generating subcategory to `orient' the theory. We have not proved that the uniqueness statement applies also to the subcategory of \emph{complete} objects --- that is, to enriched category theory proper --- but it seems likely that it holds without modification.

\

\noindent Via the approach to defining $(\infty,n)$-categories by iterated enrichment, one can directly compare these results: see \cite[\S7]{Rune_rectification} for the Gepner-Haugseng model and \cite[\S5]{Hin} for Hinich. Does the diagram
\[\xymatrix{\rm{Alg_{cat}^{GH}}(n\Cat)\ar@{<->}[dr]|{\text{\cite[Thm.\ 7.5]{Rune_rectification}}}  \ar@{<->}[rr]^-{\text{Theorem \ref{one}}} && \rm{Alg}(\rm{Quiv}_-(n\Cat))\ar@{<->}[dl]|-{\text{\cite[Cor.\ 5.6.1]{Hin}}} \\
& \rm{Seg}(n\Cat)
}\]
commute? Here the authors' approaches diverge significantly, and it would be interesting to obtain an isomorphism between these two functors.

\

\noindent The uniqueness statement becomes false if we restrict to $\sh V$ \emph{symmetric} monoidal: there is a non-trivial symmetry given by formation of the opposite category. It seems likely that this is the entire symmetry group; however, a proof of this statement does not seem to be immediately accessible to the methods of this paper.

\paragraph{Bordism description of $\mathtt{Ass}_X$}
The definition of $\sh O_X$ makes use of simplicial category model in an essential way; meanwhile, although $\sh F$ can be defined without reference to a set-theoretic model, its definition uses laborious explicit combinatorics. Only $L_\rm{gen}\Delta_X^\op$ has a truly universal feel, though this too is unsatisfactory as it only gives the correct object when $X$ is a space. Is there a holistic approach to constructing the adjunction $\sh F\dashv \sh O$?

Here is a sketch of how $\mathtt{Ass}_X$ may be defined using bordism theory. First, we recall that the associative operad $\mathtt{Ass}=\Delta^\op$ has a realisation as a bordism category in which:
\begin{itemize}
\item objects are finite disjoint unions of embedded intervals in the line $\R_x$;
\item morphisms from $X_0$ to $X_1$ are surfaces $\Sigma$ (with corners) embedded in the plane $\R_x\times[0,1]_t$, transverse at $\{0,1\}$, with identifications $\partial_i\Sigma:=\Sigma\cap\R_x\times\{i\}\cong X_i$. The surfaces must be simply-connected and $\pi_0(\partial_1\Sigma)\rightarrow\pi_0(\Sigma)$ bijective.
\end{itemize}
As usual, composition is defined by glueing surfaces at marked ends. We intro

For a given such surface $\Sigma$, let us call \emph{horizontal} the part of the boundary not contained in $\R_x\times\{0,1\}$. Then $\mathtt{Ass}_X$ will be the operad consisting of 1-manifolds as above whose boundary points are labelled with objects of $X$, and whose morphisms are embedded surfaces with horizontal boundary marked with morphisms of $X$. Its structural morphism $\mathtt{Ass}_X\rightarrow\mathtt{Ass}$ obtained by forgetting the labelling. 

A look at the pictures in \cite{Hin}, or in \S\ref{TECH} of this paper, will confirm the equivalence of this description. A full development would be, I think, best left as another story for another day.

\paragraph{Outline of the paper}In \S\ref{PRELIM} we go over general conventions and review the relevant material from \cite{Gepner_Haugseng,Hin}. In \S\ref{QUILL} we make some reductions to the case of 1-categories; this section addresses the matter of extensions of adjunctions, and model-categorical issues (especially those concerning $s\Cat_\mathtt{Ass}$). In \S\ref{TECH} we carry out the main argument --- this is conceptually straightforward, but tedious in praxis. Finally, \S\ref{UNIQUE} addresses the symmetries of $\sh F$.

\paragraph{Acknowledgements}I thank David Gepner and Rune Haugseng and Aaron Mazel-Gee for illuminating conversations on the subject; in particular, I thank Rune for drawing my attention to the work of V.\ Hinich.

\section{Preliminaries}\label{PRELIM}

\begin{para}[Category theory]This paper is based on the $(\infty,1)$-category theory developed in \cite{HTT,HA}. Hence categories, by default, are $(\infty,1)$-categories, while classical $(1,1)$-categories whose mapping objects are sets are said to be \emph{1-truncated} categories. The $(\infty,1)$-category of $(\infty,1)$-categories is denoted $\Cat$; the full $(2,1)$-subcategory of $(1,1)$-categories is denoted $\Cat^{\leq1}$. 

The category of monoidal categories is denoted $\Cat^\otimes$; the category of (planar) operads is $\Op$.

Since one of the objects being compared is defined as a simplicial category, it was not possible to entirely avoid model category techniques. The methods we actually use are quite restricted:
\begin{itemize}\item The notion of Quillen adjunction and the fact, proved in \cite{MGq}, that it induces an adjunction between the associated $(\infty,1)$-categories.
\item The Bergner model structure on the $(1,1)$-category $s\Cat$ of simplicially enriched categories.\footnote{Nowadays, it is also possible to consider this as a model $(2,1)$-category.} We also need some facts about the slice model structure; see \S\ref{QUILL}.
\end{itemize}
\end{para}

\begin{para}[Membership declaration]Objects of categories (either variables or constants) are declared with a \emph{colon}, i.e.\ $a:A$ states that $a$ is an object of $A$ (hence is equivalent to the usual notation $a\in A$). As usual, we may also write $f:A\rightarrow B$ or $g:C\cong D$ to declare a function or isomorphism; the category in which this morphism or isomorphism lives, if unclear, is underset.
\end{para}

\begin{para}[Approaches to enriched categories]There have been various attempts to get enriched category theory working using model categories. The earliest attempts used classical enrichment compatible with a model structure; this approach continues to suffer from the disadvantages encountered in the special case of simplicially enriched categories, and it is impractical for addressing $(\infty,n)$-category theory. More recently, a more general notion of `weak' enrichment \cite{Simpson}. For an overview, see the introduction to \cite{Gepner_Haugseng}.

The notion of a $\sh V$-enriched category, for $(\sh V,\otimes)$ a general monoidal $\infty$-category, was introduced in \cite{Gepner_Haugseng}. The more flexible language of \cite{Hin}, although it appeared later, was apparently developed concurrently. It is worth mentioning that Hinich proves a Yoneda lemma \cite[Cor.\ 6.2.7]{Hin}. 

In both cases the definition of the 1-category $\sh V\Cat$ of $\sh V$-enriched categories proceeds in two stages:
\begin{enumerate}\item First we define a notion which in this paper we will call an \emph{algebroid} in $\sh V$ over a space $X$, encoding the $\sh V$-valued Hom-bifunctor $X\times X\rightarrow\sh V$ with its associative composition law. (These are the objects called `categorical algebras' in \cite{Gepner_Haugseng} and `enriched precategory' in \cite{Hin}.)
\item Second, localise to a full subcategory of \emph{complete} (or \emph{univalent}) algebroids. These are those algebroids for which the underlying space $X$ \emph{classifies objects} \cite[\S5.2]{Gepner_Haugseng}.\end{enumerate}
In this paper, we will not go into details on the localisation stage.
\end{para}

\begin{para}[Corepresenting algebroids]\label{PRELIM_ALGEBROID}To a monoidal category $\sh V$ and space $X$ each of the references \cite{Gepner_Haugseng, Hin} defines the category of algebroids in $\sh V$ over $X$ by constructing a corepresenting object.
\begin{itemize}\item To $X$ \cite{Gepner_Haugseng} functorially attaches a certain generalised planar operad $\Delta_X^\op\rightarrow\Delta^\op$ \cite[\S4.1]{Gepner_Haugseng}, which will corepresent algebroids with space of objects $X$. 

The authors then define \emph{categorical algebras} (at least for $X$ a space):
\[ \rm{Alg_{cat}}(\sh V)=\rm{Alg^{GH}_{cat}}(\sh V):=\bf{Op}^\rm{gen}_\rm{planar}(\Delta^\op_X,\sh V^\otimes)\]
as a category of morphisms between generalised planar operads. 

Since the target is a monoidal category, it makes no difference if we replace $\Delta_X^\op$ with its nearest (non-generalised) operad quotient $L_\rm{gen}\Delta^\op_X$. When $X$ is presented as a simplicial groupoid, this is modelled by a simplicial multicategory $\sh O_X$ \cite[\S4.2]{Gepner_Haugseng}.

\item The associative operad is denoted $\mathtt{Ass}(=\Delta^\op)$. Hinich associates to $X$ the planar operad $\mathtt{Ass}_X$ \cite[(3.2.8)]{Hin}. The category of $\sh V$-enriched precategories is then defined as the category of operad morphisms $\bf{Op}(\mathtt{Ass}_X,\sh V^\otimes)$.
\end{itemize}
It is reasonable to regard $\mathtt{Ass}_X$ as \emph{always} corepresenting algebroids over $X$, while $\Delta_X^\op$ is correct only when $X$ is a space (some of the arrows end up pointing the wrong way in general).

Being a mapping category, as $X$ and $\sh V$ varies, it defines a bifunctor \cite[(3.5.3)]{Hin}
\[ \rm{Algbrd}:1\Cat^\op \times \Cat^\otimes \rightarrow 1\Cat. \] 
Fixing $\sh V:\Cat^\otimes$ and integrating over $1\Cat$ yields a Cartesian fibration
\[ \mathrm{Algbrd}(\sh V) \rightarrow 1\Cat \]
whose total space is by definition the category of all algebroids in $\sh V$.
\end{para}

\begin{remark}[A clarification]A couple of remarks are required to fully equate the language of \cite{Hin} with the form presented in \ref{PRELIM_ALGEBROID}.
\begin{itemize}\item The paper is written in a generalised setting of operads over a `strong approximation' to a (symmetric) operad \cite[(2.5.2)]{Hin}. In this language, $\mathtt{Ass}=\Delta^\op$ is a strong approximation to the associative operad, and $\bf{Op}_\mathtt{Ass}$ is (definitionally) the same as the category of non-symmetric operads appearing in \cite[\S3]{Gepner_Haugseng}.

\item Rather than defining them directly as a mapping category, Hinich actually defines $\sh V$-enriched precategories as algebras in the `Day convolution' internal mapping operad 
\[\rm{Quiv}_X(\sh V) = \rm{Funop}(\mathtt{Ass}_X,\sh V^\otimes)\]
from $\mathtt{Ass}_X$ to $(\sh V,\otimes)$. This object is defined for so-called \emph{flat} operads, a class which includes $\mathtt{Ass}_X$ \cite[\S3.3]{Hin}. By \cite[Cor.\ 2.6.5]{Hin}, algebras in this category compute operad morphisms, so this agrees with the formula written here.\end{itemize}\end{remark}

\section{A Quillen adjunction}\label{QUILL}

There are a couple of technical matters to address before the main statement \ref{QUILLEN_ADJOINT_SCAT} can be formulated:
\begin{itemize}\item Since $\sh F$ is defined only on $\Delta_\mathtt{Ass}$ and not $\Cat_\mathtt{Ass}$, we will need a way to exchange adjunction data on this subcategory with genuine adjunctions involving left Kan extensions of $\sh F$.

\item We need to address the slice of the Bergner model structure on $s\Cat_\mathtt{Ass}$ and show that it presents $\Cat_\mathtt{Ass}$.\end{itemize}

\begin{defn}[Formal adjunction]Let $\rm C\subseteq\rm C^\wedge$ and $\rm D$ be categories, $\rm C\stackrel{L}{\rightarrow}\rm D\stackrel{R}{\rightarrow}\rm C^\wedge$ two functors. A \emph{formal adjunction} $L\dashv R$ between $L$ and $R$ is the data of a natural isomorphism
\[ c:\rm C,\ d:\rm D \qquad \rm C^\wedge(c,Rd)\cong \rm D(Lc,d). \]
By the Yoneda lemma, $L$ is uniquely determined by $R$; if $\rm C$ is dense in $\rm C^\wedge$, the converse also holds. 
\end{defn}

\begin{lemma}[Extending a formal adjunction]\label{CAT_ENRICH_ADJOINT_LEMMA}Let $\rm{adj}:L\dashv R$, $\rm C\rightarrow\rm D\rightarrow\rm C^\wedge$ be a formal adjunction. Suppose that $L$ admits a (pointwise) left Kan extension $L^\wedge:\rm C^\wedge\rightarrow\rm D$. Then there is a unique adjunction $L^\wedge\dashv R$ extending the given formal adjunction. In particular, $\Aut(L)=\Aut(L^\wedge)$.\end{lemma}
\begin{proof}Immediate from the definitions.\end{proof}

\begin{para}[$\sh F$]
We recall that Hinich's functor $\sh F$ is, by definition, formally left adjoint to $\mathtt{Ass}$, as in:
\[\xymatrix{ \Delta_\mathtt{Ass}\ar@{>->}[d] \ar[r] & \Cat\ar[dl] \\ \Cat_\mathtt{Ass}. }\]
Moreover, $\sh F$ actually factors through the subcategory $\Cat^{\leq1}\subset\Cat$ of classical (1-truncated) categories.
\end{para}

\begin{para}[Model structure on a slice]The slice category $s\Cat_\mathtt{Ass}$ inherits sets of fibrations, cofibrations, and weak equivalences from the corresponding sets in the Bergner model structure by inverse image along the forgetful functor $s\Cat_\mathtt{Ass}\rightarrow s\Cat$. It is well-known (and straightforward to see) that this is also a model structure. By lemma \ref{QUILLEN_SLICE} below, its localisation is the $\infty$-category $\Cat_\mathtt{Ass}$.\footnote{The lemma is surely well-known, but it seems easier to prove it than to locate a reference.}
\end{para}

\begin{lemma}[Slicing over a local object commutes with localisation]\label{QUILLEN_SLICE}Let $(\rm C,W)$ be a relative category, $\rm C\rightarrow\rm C[W^{-1}]$ an ($\infty$-categorical) localisation. Let $X:\rm C$ be an object, and let $W_X\subseteq\rm C_X$ be the class of morphisms inherited from $W$ under the forgetful functor $\rm C_X\rightarrow\rm C$.

If $X$ is $W$-local, then the natural functor $\rm C_X\rightarrow \rm C[W^{-1}]_X$ is a localisation of $\rm C_X$ at $W_X$.\end{lemma}
\begin{proof}We will show that the categories of $W$-local presheaves are equivalent by realising them as Cartesian fibrations in spaces.

First, composition $\Cart_0(\rm C_X)\rightarrow\Cart_0(\rm C)_X$ with the forgetful functor is an equivalence of categories: indeed, a morphism $\rm E\rightarrow \rm C_X$ of Cartesian fibrations in spaces over $\rm C$ is itself automatically a Cartesian fibration in spaces.

Second, this equivalence exchanges the objects whose restriction to $W$ is a local system. Let $u\rightarrow v$ be in $W$. There is an induced commuting square
\[\xymatrix{ \rm E_v\ar[r]\ar[d] & \rm E_u\ar[d] \\
\rm C(v,x)\ar[r] & \rm C(u,x) }\]
whose lower horizontal arrow is invertible since $x$ is local. Thus the upper arrow is invertible if and only if its restriction to each fibre is invertible.
\end{proof}

\begin{para}[$\sh O_X$]
Gepner-Haugseng's $\sh O_X$ is a functor valued in simplicial multicategories. Throughout, we tacitly replace this multicategory with its operad of operators (reviewed in \cite[\S2.2]{Gepner_Haugseng}), which is a simplicial category over $\mathtt{Ass}$. 

It follows from the definition, and the fact that weak equivalences are preserved by products, that $\sh O_X$ descends to a functor $\Cat\rightarrow\Cat_\mathtt{Ass}$.\end{para}

We are now ready to formulate our main reduction step in the proof of theorem \ref{one}.

\begin{prop}[$\sh F\dashv\sh O$]\label{QUILLEN_ADJOINT_SCAT}Suppose given a formal adjunction $\rm{adj}:\sh F\dashv\sh O$ beetween $\sh F$ and the restriction of $\sh O$ to the category $\Cat^{\leq0}$ of classical 1-categories.

Then there is a unique extension of $\rm{adj}$ to a Quillen adjunction
\[ \hat{\sh F}_\bullet:s\Cat_\mathtt{Ass} \rightleftarrows s\Cat:\sh O, \]
between the Bergner model structure on $s\Cat$ and the slice structure on $s\Cat_\mathtt{Ass}$.\end{prop}
\begin{proof}We define the functor $\hat{\sh F}_\bullet:s\Cat_\mathtt{Ass}\rightarrow s\Cat$ by applying $\hat{\sh F}$ in each dimension. (Note that $\sh O$ is also defined this way, as products of simplicial sets are calculated dimension-by-dimension.) Then mapping sets are computed using an end:
\begin{align}\Fun_\mathtt{Ass}(S_\bullet,(\sh O_X)_\bullet) &= \rm{end}_{n,m}\Fun_\mathtt{Ass}(S_m,\sh O_{X_n}) \\
&= \rm{end}_{n,m}\Fun(\hat{\sh F}(S_m),X_n) \\
&= \Fun(\hat{\sh F}_\bullet S_\bullet, X_\bullet)\end{align}
so that $\hat{\sh F}_\bullet\dashv \sh O_\bullet$.

Moreover, $\sh O$ preserves weak equivalences and fibrations because weak equivalences and fibrations of simplicial sets are preserved by products and the mapping spaces in $\sh O_X$ are products of those in $X$. In other words, it is a right Quillen functor, so $\hat{\sh F}\dashv\sh O$ is a Quillen adjunction.
\end{proof}

\begin{cor}[$\sh F\dashv\sh O$]\label{QUILLEN_MAIN}Suppose given a formal adjunction $\rm{adj}:\sh F\dashv\sh O$ beetween $\sh F$ and the restriction of $\sh O$ to the category $\Cat^{\leq0}$ of classical 1-categories. There is a unique extension to an adjunction
\[ \hat{\sh F}:1\Cat_{\Delta^\op} \rightleftarrows 1\Cat:\sh O \]
between $\sh O$ and a left Kan extension $\hat{\sh F}$ of $\sh F$.\end{cor}
\begin{proof}By \cite{MGq}, the derived functors of a Quillen adjunction yield an adjunction of the localised $\infty$-categories. Here we used lemma \ref{QUILLEN_SLICE} to conclude that $s\Cat_\mathtt{Ass}$ is indeed a model for $\Cat_\mathtt{Ass}$.
\end{proof}

\section{The technical bit}\label{TECH}

In this section, we will construct the adjunction at the level of 1-categories. This comparison is straightforward, and it is foreshadowed in \cite[(3.2.9)]{Hin}, but sadly the specifics are rather fiddly.

\begin{prop}[$\sh F\dashv\sh O$]\label{TECH_ADJOINT}There is a formal adjunction
\[\xymatrix{ \Delta_\mathtt{Ass}\ar@{>->}[d]\ar[r]^-{\sh F} & \Cat^{\leq1}\ar[dl]^-{\sh O} \\
\Cat^{\leq1}_\mathtt{Ass} }\]
between Hinich's functor $\sh F$ and the restriction of Gepner-Haugseng's $\sh O$ to classical 1-categories.\end{prop}

\begin{para}[Notation]\label{TECH_NOTATION}In what follows, we denote objects of $\mathtt{Ass}=\Delta^\op$ in the format $[n]$, corresponding to the $n$-simplex $\Delta^n:\Delta$ in the opposite category. If $\sigma:\Delta^k\rightarrow\mathtt{Ass}$ is a $k$-simplex of $\mathtt{Ass}$, we write $\sigma_i$ for its $i$th vertex, so $\sigma_i=[n]$ for some $n:\N$.

\

\noindent\emph{Inert}. Let us write $[k]\subseteq[n]$ when $[k]$ is a convex subset of the totally ordered set $[n]$ (dual to an \emph{inert} morphism $[n]\rightarrow[k]$ in $\mathtt{Ass}$ \cite[Def.\ 3.1.1]{Gepner_Haugseng}). If $f:[m]\rightarrow[n]$ is a morphism in $\mathtt{Ass}$, there is an induced pullback operation $f^{-1}$ on the poset of convex subsets. It is constructed by taking the convex hull of the image under the dual map $\Delta^n\rightarrow\Delta^m$. 

\

\noindent\emph{Decomposition}. In particular, the $n$ inclusions $i:[1]\subseteq[n]$ induce a decomposition of $[m]$ as a concatenation $ \rm{cat}_{i:[1]\subseteq[n]}[m_i]=f^{-1}i$. This decomposition is used to formulate the key axiom in the definition of a planar operad \cite[Def.\ 3.1.3, (ii)]{Gepner_Haugseng}.
\end{para}

\begin{para}[Graphs]Denote by $\Delta^{[0,1]}\subset\Delta$ the full subcategory spanned by $\Delta^0$ and $\Delta^1$. An \emph{oriented graph} $\Gamma$ defines a presheaf on $\Delta^{[0,1]}$ by defining $\Map(\Delta^k,\Gamma)$ to be the set of vertices of $\Gamma$ if $k=0$ and the union of the sets of (oriented) edges and of vertices for $k=1$. The graph may be recovered from its associated presheaf: the vertices are the 0-simplices, the edges are the non-degenerate 1-simplices, and the face maps yield the incidence relation. 

We define the \emph{category of oriented graphs} to be the full subcategory $\bf{Grf}\subset\PSh(\Delta^{[0,1]})$ spanned by the presheaves obtained in this way. Left Kan extension of the inclusion $\Delta^{[0,1]}\subset\Cat$ yields a colimit-preserving functor 
\[\langle-\rangle:\bf{Grf}\rightarrow\Cat,\]
the \emph{free category} functor. It lands in $\Cat^{\leq1}$.

The finite inhabited totally ordered sets may be thought of as oriented graphs with vertices the elements of the set and edges the nearest-neighbour order relations. Under this correspondence, the $n$-simplex corresponds to an $A_{n+1}$ graph. Let $\Delta^\rm{Grf}\subset\bf{Grf}$ be the full subcategory spanned by these objects. The free category functor fits into a square
\[\xymatrix{ \Delta^\rm{Grf}\ar[r]\ar[d] & \Delta\ar[d] \\ \bf{Grf}\ar[r] & \Cat.}\]
The map $\Delta^\rm{Grf}\rightarrow\Delta$ is the inclusion of a wide subcategory whose morphisms, in terms of the usual terminology for morphisms in $\Delta$, are generated by the degeneracy and \emph{outer} face maps.

We denote by $\Delta^{[0,1]}_\mathtt{Ass}\subset\Delta^\rm{Grf}_\mathtt{Ass}\subset\Delta_\mathtt{Ass}$ the corresponding comma categories.
\end{para}

\begin{para}[Summary of the definition of $\sh F$]\label{TECH_F_DEF}The definition of the functor $\sh F$ is rather involved --- it occupies 5 pages of \cite[\S3.2]{Hin} --- and since the proof of \ref{TECH_ADJOINT} will follow the same lines it will be helpful to have a summary of it here.

\begin{enumerate}
\item First, $\sh F^{[0,1]}$ is defined explicitly on the full subcategory $\Delta^{[0,1]}\subseteq\Delta$ spanned by the zero and one-simplices of $\mathtt{Ass}$:
\begin{enumerate}
\item It is defined as a map on the set $\{\Delta^0\}_\mathtt{Ass}$ of zero-simplices by the formula:
\begin{align}\label{F_def_0} \sh F^{[0,1]}(\sigma:\Delta^0\rightarrow\mathtt{Ass}):=\{[1]\subseteq\sigma_0\}\times\{\bf x,\bf y\} \end{align}
where $\{[1]\subseteq\sigma_0$ is the set (of cardinality $\#\sigma_0-1$) of edges of $\sigma_0$.

Let us denote by $\sh F_0:\Delta_\mathtt{Ass}\rightarrow\bf{Set}$ the functor that associates to $\sigma:\Delta^k\rightarrow\mathtt{Ass}$ the disjoint union of $\sh F^{[0,1]}(\tau)$, where $\tau$ ranges over all vertices of $\Delta^k$. Note that this is a left Kan extension of the restriction of $\sh F$ to the domain $\{\Delta^0\}_\mathtt{Ass}$ of \eqref{F_def_0}. In general, $\sh F(\sigma)$ will be a certain category structure on $\sh F_0(\sigma)$.

\item For the set of 1-simplices I will not reproduce the full definition \cite[(3.2.3)]{Hin}; suffice it to say that for the generating case $\sigma_1=[1]$, $\sh F^{[0,1]}(\sigma)$ is described by the diagram
\[\xymatrix{
\bf x_{01} & \bf y_{01}\ar[r] & \bf x_{12}\ar@{}[r]|-\cdots  & \bf y_{(n-2)(n-1)}\ar[r] & \bf x_{(n-1)n} & \bf y_{(n-1)n} \ar[dll] \\
&&\bf x_{01}\ar[ull] & \bf y_{01}
}\]
where the vertices of the two rows are $\sh F^{[0,1]}(\sigma_0)$ and $\sh F^{[0,1]}(\sigma_1)$, respectively. This diagram appears as \cite[(3.2.4), fig.\ (30)]{Hin}. 
More generally, for any $\sigma:\Delta^1\rightarrow\mathtt{Ass}$, the result is a finite disjoint union of categories equivalent to $\Delta^1$.

\item The two face maps of $\Delta^{[0,1]}$ are such that their coproduct is the inclusion
\begin{align}\sh F_0(\sigma)=\sh F^{[0,1]}(\sigma_0)\sqcup \sh F^{[0,1]}(\sigma_1) \hookrightarrow\sh F^{[0,1]}(\sigma) \end{align}
for any $\sigma:\Delta^1\rightarrow\mathtt{Ass}$, as was implicit in the preceding diagram. For the degeneracy map, note only that a the image of a degenerate simplex on $[n]$ is the graph
\[\xymatrix{
\bf x_{01} & \bf y_{01}\ar[d]\ar@{}[r]|-\cdots & \bf x_{(n-1)n} & \bf y_{(n-1)n}\ar[d] \\
\bf x_{01}\ar[u] & \bf y_{01}\ar@{}[r]|-\cdots & \bf x_{(n-1)n}\ar[u] & \bf y_{(n-1)n}
}\]
which maps in a unique way to $\sh F^{[0,1]}([n])$ preserving the labelling.

By illustration, the face maps are sections of the degeneracy map, that is, they obey the simplicial identities; hence these formulae define a functor $\sh F^{[0,1]}:\Delta^{[0,1]}_\mathtt{Ass}\rightarrow\bf{Grf}$.
\end{enumerate}
\item By left Kan extension we may immediately extend $\sh F^{[0,1]}$ to a functor of oriented graphs $\sh F^\rm{Grf}:\bf{Grf}_\mathtt{Ass}\rightarrow\bf{Grf}$, and in particular, of $\Delta^\rm{Grf}_\mathtt{Ass}$. Since $\sh F_0\mid\Delta^\rm{Grf}$ is also a left Kan extension, it defines a subfunctor $\sh F_0^\rm{Grf}\subset\sh F^\rm{Grf}$.

(Actually, Hinich's definition uses an explicit colimit \cite[(3.2.5)]{Hin} over the poset of cells of the graph $\Delta^k$. This poset is a reflective subcategory of the full comma category $\Delta^{[0,1]}\downarrow_\bf{Grf}\Delta^k$ with reflector given by taking the image. In particular, it is cofinal, and so this colimit does indeed compute the left Kan extension.)

\item Finally, one defines functoriality for the \emph{inner} face maps, which encode composition in $\mathtt{Ass}$ and hence in the multicategory $\mathtt{Ass}_X$. Since $\sh F(\sigma)=\sh F^\rm{Grf}(\sigma)$ is a poset, and we already have the action of inner face maps on the underlying set $\sh F_0(\sigma)$, it is merely a condition for them to be functors. This is checked \cite[Lemma 3.2.7]{Hin}.
\end{enumerate}

\end{para}

\begin{para}[Vertices and edges]We will define the natural transformation separately for each object and morphism in $\Delta^{[0,1]}$.

\

\noindent $(|\sigma|=0)$. The operad $\sh O_X$ has set of objects $\rm{Ob}X\times\rm{Ob}X$ which we identify with $\Map(\{\bf x,\bf y\},X)$ in the order implied by the orthography. 
\begin{align}\label{adj _0} \Fun_\mathtt{Ass}(\sigma,\sh O_X) &= \rm{Ob}(\sh O_X(\sigma_0)) \\
&=(X\times X)^{\{[1]\subseteq\sigma_0\}} \\
&= \Fun(\{[1]\subseteq\sigma_0\}\times\{\bf x,\bf y\},X)
\end{align}
(Note that there is exactly one other way to identify these spaces naturally in $X$, which simply reverses the factors. The arguments of \S\ref{UNIQUE} show that it is not possible to make this identification compatible with edges.)

\

\noindent $(|\sigma|=1)$. A 1-simplex $\sigma:\Delta^1\rightarrow\mathtt{Ass}$ is the data of a morphism $\sigma_0\rightarrow\sigma_1$ in $\mathtt{Ass}$. As in the definition \ref{TECH_F_DEF}-i)-(b) of $\sh F$, we separate the case that $\sigma_1=[1]$.

\begin{itemize}
\item $(\sigma_1=[1])$. The equivalence of $\Map(\sh F(\sigma),X)$ with \cite[Def.\ 4.2.4]{Gepner_Haugseng} is defined by the labelling
\begin{align}\label{diagram_label}\xymatrix{x_0 & y_1\ar[r] & x_1 \ar@{}[r]|\cdots & y_{n-1}\ar[r] & x_{n-1} & y_n\ar[dll] \\
&& y_0\ar[ull] & x_n}\end{align}
of Hinich's diagram with Gepner-Haugseng's variables: the set of such diagrams is precisely
\begin{align}\label{mapping_label} \sh O_X(\underbrace{(x_0,y_1)}_{u_1},\ldots,\underbrace{(x_{n-1},y_n)}_{u_n};\ \underbrace{(y_0,x_n)}_{v})
&= X(y_0,x_0)\times\cdots \times X(y_n,x_n)\end{align}
as a functor of $\{(x_i,y_i)\}_{i=0}^n=(u,v)$, that it is the space of polymorphisms of $\sh O_X$. 

Now integrating over $(X^\op\times X)^n$,
\begin{align}\label{adj_11} \rm{adj}:\Fun_\mathtt{Ass}(\sigma,\sh O_X) & = \int_{u:\sh O_X(\sigma_1)}\int_{v:\sh O_X(1)} \sh O_X(u,v)\\
&= \int_{u:(X^\op\times X)^n}\int_{v:X^\op\times X}\prod_{i=0}^n X(y_i,x_i) & \text{by def.}\ \eqref{mapping_label}\\
&= \Fun(\sh F(\sigma),X). & \text{by labelling }\eqref{diagram_label}
\end{align}
(For reasons of formatting I have found it necessary to confuse Gepner-Haugseng's variables with my own, underset in \eqref{mapping_label}. Fully expanded, the definition of the polymorphism set in this notation is as follows:
\begin{align}
\sh O_X\left((u^\bf x_1,u^\bf y_1),\ldots,(u^\bf x_n, u^\bf y_n);\ (v^\bf x,v^\bf y) \right)
&= X\left(v^\bf x,u^\bf x_1\right)\times\prod_{i=1}^nX\left(u^\bf y_{i-1},u^\bf x_i\right) \times X\left(u^\bf y_n,v^\bf y\right).
\end{align}
Beware that $x$s and $y$s do not occur in the same order as Hinich's $\bf x$s and $\bf y$s as described in definition \ref{TECH_F_DEF}-i)-(c).)

\item (General case). As explained in \ref{TECH_NOTATION}, to each segment $i:[1]\subseteq \sigma_1$ there corresponds by pullback a segment $\sigma_0[i]\subseteq\sigma_0$, and $\sigma_0=\rm{cat}_{i:[1]\subseteq\sigma_1}\sigma_0[i]$. Denote by $\sigma[i]$ the corresponding arrow $[m_i]\rightarrow[1]$ so that $\sigma$ is the concatenation of the $\sigma[i]$ with respect to the concatenation operation on $\mathtt{Ass}$.

By definition,
\begin{align}\label{F_GEN}\sh F(\sigma):= \coprod_{i:[1]\subseteq \sigma_1} \sh F(\sigma[i])\end{align}
and so we calculate:
\begin{align} \Fun_\mathtt{Ass}(\sigma,\sh O_X)&=\int_{u:\sh O_X(\sigma_0)}\int_{v:\sh O_X(\sigma_1)}\sh O_X(u,v) \\
&=\int_{u:\sh O_X(\sigma_0)}\int_{v:\sh O_X(\sigma_1)} \prod_{i:[1]\subseteq\sigma_1}\sh O_X(u_i,v_i) & 
\sh O_X\text{ is an operad}\\
\label{prod=int}&=\prod_{i:[1]\subseteq\sigma_1}\int_{u:\sh O_X((\sigma_0)_i)}\int_{v:\sh O_X(1)} \sh O_X(u,v) & 
\textstyle\int\leftrightarrow \prod \\
&= \prod_{i=1}^n \Fun(\sh F(\sigma_i),X) & \text{via}\ \rm{adj}\ \eqref{adj_11}\\
&= \Fun(\sh F(\sigma),X). & \text{by def.}\ \eqref{F_GEN}
\end{align}\end{itemize}

\

\noindent\emph{Face maps}. Naturality with respect to face maps is an automatic consequence of the fact that $\sh F_0\subset\sh F$ is a subfunctor.

\

\noindent\emph{Degeneracy}. Since the operad structure gives us a decomposition
\[\xymatrix{
\Fun_\mathtt{Ass}(\sigma,\sh O_X) \ar@{=}[r]\ar[d] & (X^\op\times X)^{\{[1]\subseteq\sigma_0\}}\ar[d]\\
\Fun_\mathtt{Ass}(\delta\sigma,\sh O_X)\ar@{=}[r] & (X^\op\times X)^{\{[1]\subseteq\sigma_0\}\times \Delta^1}
}\]
it suffices to check that the square
\[\xymatrix{
X^\op\times X\ar[d] \ar[r]^-{\rm{adj}_0} & \Fun(\sh F(\sigma),X)  \ar[d]\ar[l]\\
(X^\op\times X)^{\Delta^1} \ar[r]^-{\rm{adj}_1}  &  \Fun(\sh F(\delta\sigma),X)\ar[l]
}\]
commutes in the case $\sigma_0=[1]$. But then, applying $\rm{adj_0}$ to $u=(u^\bf x,u^\bf y)$ yields the diagram
\[\xymatrix{ u^\bf x & u^\bf y\ar[d]^\iden \\ u^\bf x\ar[u]^{\iden} & u^\bf y }\]
which the reader will readily see is a special case of the labelling \eqref{diagram_label} applied to $\sh O_X(u,u)=\sh O_X((u^\bf x,u^\bf y),(u^\bf x,u^\bf y))$.\end{para}

\begin{remark}I comment here on the exchange of products with Grothendieck integral that occurs in line \eqref{prod=int}. We have used here the commutativity of a diagram
\[\xymatrix{ \prod_i\Fun(\rm C_i,\bf S)\ar[r]\ar[d] & \prod_i\Cart_0(\rm C_i)\ar[d] \\
\Fun(\prod_i\rm C_i,\bf S) \ar[r] & \Cart_0(\rm C_i) }\]
I don't know a reference for this fact, but it can be deduced rather easily from the naturality of the construction in $\rm C$ and the fact that $\int$ preserves products (and indeed, all limits).\end{remark}

\begin{para}[Inner face maps]For our adjunct isomorphism $\rm{adj}$ to be natural in $\Delta_\mathtt{Ass}$, rather than merely $\Delta^\rm{Grf}_\mathtt{Ass}$, is a condition: that the square
\[\xymatrix{ 
\Fun_\mathtt{Ass}(\sigma,\sh O_X)\ar@{<->}[r]\ar[d]_\circ & \Fun(\sh F(\sigma),X)\ar[d]\\
\Fun_\mathtt{Ass}(\sigma,\sh O_X)\ar@{<->}[r] & \Fun(\sh F(\tau),X)
 }\]
induced by the inner face map $\Delta^1\rightarrow\Delta^2$ is commutative. The left-hand vertical arrow here is the composition law in $\sh O_X$. Since \cite[Def.\ 4.2.4]{Gepner_Haugseng} is not absolutely explicit about the definition of this composition law, the most I can show here is that what they say corresponds intuitively to the behaviour of the right-hand vertical arrow.

The functor $\sh F(\tau)\rightarrow\langle\sh F(\sigma)\rangle$ sends each morphism of $\sh F(\tau)$ to a (unique) composition of edges of the graph $\sh F(\sigma)$. It is reasonable to suppose that this is what Gepner-Haugseng meant by defining composition "in the obvious way, using composition in $X$" --- indeed, since the diagrams involved are posets, it is the \emph{only} possible formula in terms of the composition law of $X$ that maps from $X^\sigma$ to $X^\tau$.

For an idea of precisely which compositions occur, consult the pictures in \cite[(3.2.6)]{Hin}.
\end{para}

\section{Unicity}\label{UNIQUE}
In this section, we will show that the automorphism group of $\sh F$ is trivial. We continue to use the conventions \ref{TECH_NOTATION}, \ref{TECH_F_DEF} outlined in \S\ref{TECH}.

\begin{prop}\label{UNIQUE_PROP}The automorphism group of the composite of $\sh F$ with groupoid completion $B\sh F:\Delta_\mathtt{Ass}\rightarrow\bf S$ is trivial.\end{prop}
\begin{proof}Fix $\phi:\Aut(\sh F)$. By naturality with respect to the face maps $\Delta^0\rightarrow\Delta^k$, the action of $\phi$ on $\sh F(\sigma)$ must preserve the decomposition of the set of objects as $\coprod_{i:\Delta^k}\sh F(\sigma_i)$. Since by \cite[Lemma 3.2.5]{Hin} each $\sh F(\sigma)$ is simply-laced, if $\phi$ is the identity on each 0-simplex, then it is the identity. Thus, it suffices to show that $\phi$ acts trivially on $\sh F$ of 0-simplices.

\

\noindent \emph{Decomposition into segments}. Each segment $[1]\subseteq \sigma_0$ yields a graph \[\xymatrix{& & \bf x \ar@{-}[d] & \bf y \ar@{-}[d] \\
\cdots & \bf y_{i-1} & \bf x_i & \bf y_i & \bf x_{i+1} & \cdots}\qedhere\]
from which it may be seen that $\phi$ must preserve each pair $\{\bf x_i,\bf y_i\}$, and moreover act on that pair through $\phi\mid\sh F([1])$.

\

\noindent ($\phi\mid\sh F([1])$). It remains to show that $\phi\mid\sh F([1])$ is necessarily trivial. This follows from the diagram
\[\xymatrix{
\bf x_1 & \bf y_1\ar@{-}[r] & \bf x_2 & \bf y_2 \\
& \bf x\ar@{-}[ul] & \bf y\ar@{-}[ur]
}\]
associated by $\sh F$ to the 1-simplex $[2]\rightarrow[1]$ given by the inner face map: no symmetry of this diagram preserves both the decomposition into top and bottom rows and of the top row into two pairs.
\end{proof}

\begin{remark}[Opposite]The formation of opposites as discussed in \cite[\S2.9]{Hin} has the following expression in the variables of the present section:
\begin{align} \sh O_{X^\op} \left[ (u_1^\bf x,u_1^\bf y),\ldots,(u_n^\bf x,u_n^\bf y);\ (v^\bf x,v^\bf y) \right]
 &= X^\op(v^\bf x, u_1^\bf x)\times \prod_{i=1}^{n-1}X^\op(u_i^\bf y,u_{i+1}^\bf x)\times X^\op(u_n^\bf y,v^\bf y) \\
&=  X(v^\bf y,u_n^\bf y)\times \prod_{i=1}^{n-1}X(u_{n-i+1}^\bf x,u_{n-i}^\bf y)\times X(u_1^\bf x,v^\bf x) \\
&= \sh O_X\left[ (u_n^\bf y,u_n^\bf x),\ \ldots,\ (u_1^\bf y,u_1^\bf x);\ (v^\bf y,v^\bf x) \right] \\
&= \sh O^\rm{rev}_X\left[ (u_1^\bf y,u_1^\bf x),\ \ldots,\ (u_n^\bf y,u_n^\bf x);\ (v^\bf y,v^\bf x) \right]
\end{align}
where $\sh O^\rm{rev}$ denotes the reversed operad. In other words, the right Quillen functor $\sh O$ can be made $\Z/2\Z$-equivariant with respect to the action by opposite on $s\Cat$ and reversal on $s\Cat_\mathtt{Ass}$. Correspondingly, the $\infty$-adjunction $\sh F\dashv \mathtt{Ass}$ is also $\Z/2\Z$-equivariant. Since the adjunction has no automorphisms, this equivariance is even unique.\end{remark}

\printbibliography
\end{document}

%% file: header.tex
\usepackage[dvipsnames]{xcolor}
\usepackage[all,2cell,cmtip]{xy}
\UseTwocells
\entrymodifiers={+!!<0pt,\fontdimen22\textfont2>}

\usepackage{fouriernc}
\usepackage[english]{babel}         
\usepackage{verbatim}               
\usepackage[T1]{fontenc}
\usepackage[utf8]{inputenc}

\usepackage{amsmath}
\usepackage{amsfonts}
\usepackage{amssymb}
\usepackage{mathrsfs}    
\usepackage{amsthm}

\usepackage{todonotes}

\usepackage[style=alphabetic,backend=bibtex]{biblatex}
\usepackage[colorlinks=true,linkcolor=MidnightBlue,citecolor=OliveGreen]{hyperref}           

\swapnumbers

\newtheoremstyle{mlad} 
    {\topsep}                    
    {\topsep}                    
    {}                   
    {}                           
    {\bfseries}                   
    {.}                          
    {.5em}                       
    {}  

\theoremstyle{plain}

\newtheorem{thm}{Theorem}[section]
\newtheorem{prop}[thm]{Proposition}
\newtheorem{cor}[thm]{Corollary}

\newtheorem{lemma}[thm]{Lemma}

\theoremstyle{definition}

\newtheorem{para}[thm]{}
\newtheorem{defn}[thm]{Definition}

\theoremstyle{remark}
\newtheorem{remark}[thm]{Aside}

\newcommand{\proofof}[1]{\end{#1}\begin{proof}}

\makeatletter
\renewcommand\section{\@startsection {section}{1}{\z@}%
  {-3.5ex \@plus -1ex \@minus -.2ex}{2.3ex \@plus.2ex}%
  {\normalfont\large\bfseries}}
\renewcommand\subsection{\@startsection{subsection}{2}{\z@}%
  {-3.25ex\@plus -1ex \@minus -.2ex}{1.5ex \@plus .2ex}%
  {\normalfont\bfseries}}

\newcommand{\sh}[1]{\mathcal{#1}}
\newcommand{\N}{{\mathbb N}}
\newcommand{\Z}{{\mathbb Z}}

\newcommand{\R}{{\mathbb R}}

\newcommand{\iden}{\mathrm{id}}

\DeclareMathAlphabet{\mathrmsl}{OT1}{cmr}{m}{sl}

\newcommand{\rssymb}[2]{\newcommand{#1}{\mathrmsl{#2}} }
\newcommand{\oper}[3][n]{\newcommand{#2}{\mathop{\mathrm{#3}}%
\ifx n#1\nolimits\else\limits\fi} }
\newcommand{\rsoper}[3][n]{\newcommand{#2}{\mathop{\mathrmsl{#3}}%
\ifx n#1\nolimits\else\limits\fi} }

\renewcommand{\bf}[1]{\mathbf{#1}}
\renewcommand{\rm}[1]{\mathrm{#1}}

\oper\Ad{Ad}
\oper\ad{ad}

\oper\val{val}
\oper\coker{coker}
\oper\mult{mult}
\oper\Iso{Iso}
\oper\End{End}
\oper\Aut{Aut}
\oper\Sub{Sub}
\oper\Alt{Alt}
\oper\Ext{Ext}
\oper\Pic {Pic}
\oper\Sym{Sym}
\oper\Spec{Spec}
\oper\Spf{Spf}
\oper\Sp{Sp}
\oper\Spa{Spa}
\oper\Proj{Proj}
\rsoper\divg{div}
\rsoper{\sym}{sym}
\rsoper{\alt}{alt}
\rsoper\trace{tr}
\rssymb\id{id}

\newcommand{\thismonth}{\ifcase\month\or
  January\or February\or March\or April\or May\or June\or
  July\or August\or September\or October\or November\or December\fi
  \space\number\year}

